\documentclass[a4paper,11pt]{article}

\usepackage[latin1]{inputenc}
\usepackage[english]{babel}
\usepackage{graphicx}
\usepackage{amssymb}
\usepackage{amsmath}
\usepackage{amsfonts}
\usepackage{amsthm}

\newtheorem{theorem}{Theorem}

\newtheorem{algorithm}{Algorithm}
\newtheorem{corollary}{Corollary}

\newtheorem{example}{Example}
\newtheorem{lemma}{Lemma}

\newtheorem{remark}{Remark}
\numberwithin{equation}{section}

\usepackage{algorithm}
\usepackage{algorithmic}

\title{Nonnegative Factorization and \\The Maximum Edge Biclique Problem} 

\author{Nicolas Gillis \and Fran\c{c}ois Glineur\thanks{Center for Operations Research and Econometrics, Universit\'e catholique de Louvain, Voie du Roman Pays, 34, B-1348 Louvain-La-Neuve, Belgium ;
nicolas.gillis@uclouvain.be and francois.\mbox{glineur}@uclouvain.be. The first author is a research fellow of the Fonds de la Recherche Scientifique (F.R.S.-FNRS). This text presents research results of the Belgian Program on Interuniversity Poles of Attraction initiated by the Belgian State, Prime Minister's Office, Science Policy Programming. The scientific responsibility is assumed by the authors.} }

\date{\small October 2008}

\begin{document}

\maketitle
\thispagestyle{empty}

\begin{abstract}
Nonnegative Matrix Factorization (NMF) is a data analysis technique which allows compression and interpretation of nonnegative data. NMF became widely studied
after the publication of the seminal paper by Lee and Seung (Learning the Parts of Objects by Nonnegative Matrix Factorization, Nature, 1999, vol. 401, pp. 788--791),
which introduced an algorithm based on Multiplicative Updates (MU). More recently, another class of methods called Hierarchical Alternating Least Squares (HALS) was introduced that seems to be much more efficient in practice.

In this paper, we consider the problem of approximating a not necessarily nonnegative matrix with the product of two nonnegative matrices, which we refer to as Nonnegative Factorization~(NF)~; this is the subproblem that HALS methods implicitly try to solve at each iteration. We prove that NF is NP-hard for any fixed factorization rank, using a reduction to the maximum edge biclique problem.

We also generalize the multiplicative updates to NF, which allows us to shed some light on the differences between the MU and HALS algorithms for NMF and give an explanation for the better performance of HALS. Finally, we link stationary points of NF with feasible solutions of the biclique problem to obtain a new type of biclique finding algorithm (based on MU) whose iterations have an algorithmic complexity proportional to the number of edges in the graph, and show that it performs better than comparable existing methods.
\bigskip

\end{abstract}

\newpage

\section{Introduction}

\emph{(Approximate) Nonnegative Matrix Factorization} (NMF) is the problem of approximating a given nonnegative matrix by the product of
two low-rank nonnegative matrices: given a matrix $M \geq 0$, one has to compute two low-rank matrices $V, W \geq 0$
such that
\begin{equation}
M \approx V W \;.
\label{approx}
\end{equation}
\noindent
This problem was first introduced in 1994 by Paatero and Tapper \cite{Paa}, and more recently received a considerable
interest after the publication of two papers by Lee and Seung \cite{LS1, LS2}. It is now well established that NMF is
useful in the framework of compression and interpretation of nonnegative data ; it has for example been applied in
analysis of image databases, text mining, interpretation of spectra, computational biology and many other applications (see e.g.~\cite{Ber, Dev, Dhi} and references therein).\\

\noindent How can one interpret the outcome of a NMF? Assume each column $M_{:j}$ of matrix $M$ represents an element
of a data set: expression (\ref{approx}) can be equivalently written as
\begin{equation}
M_{:j} \approx \sum_{k} V_{:k} W_{kj}, \quad \forall j
\label{inter}
\end{equation}
where each element $M_{:j}$ is decomposed into a nonnegative linear combination (with weights $W_{kj}$) of nonnegative
basis elements ($\{V_{:k}\}$, the columns of $V$). Nonnegativity of $V$ allows interpretation of the basis elements in
the same way as the original nonnegative elements in $M$, which is crucial in applications where the nonnegativity
property is a requirement (e.g. where elements are images described by pixel intensities or texts represented by
vectors of word counts). Moreover, nonnegativity of the weight matrix $W$ corresponds to an essentially additive
reconstruction which leads to a \emph{part-based representation}: basis elements will represent similar parts of the
columns of $M$. Sparsity is another important consideration: finding sparse factors improves compression and leads to
a better part-based representation of the data \cite{Hoy}.

We start this paper with a brief introduction to the NMF problem: Section~\ref{nmf} recalls existing complexity results, introduces two well-known classes of methods: multiplicative updates \cite{LS2} and hierarchical alternating least squares \cite{Cic} and proposes a simple modification to guarantee their convergence. The central problem studied in this paper, Nonnegative Factorization (NF), is a generalization of NMF where the matrix to be approximated with the product of two low-rank nonnegative matrices is not necessarily nonnegative. NF is introduced in Section~\ref{SecNF}, where it is shown to be NP-hard for any given factorization rank, using a reduction to the problem of finding a maximum edge biclique. Stationary points of the NF problem used in that reduction are also studied. This section ends with a generalization of the NMF multiplicative updates rules to the NF problem and a proof of their convergence. This allows us to shed new light on the standard multiplicative updates for NMF: a new interpretation is given in Section~\ref{interMU}, which explains the relatively poor performance of these methods and hints at possible improvements. Finally, Section~\ref{mbfa} introduces a new type of biclique finding algorithm that relies on the application of multiplicative updates to the equivalent NF problem considered earlier. This algorithm only requires a number of operations proportional to the number of edges of the graph per iteration, and is shown to perform well when compared to existing methods.

\section{Nonnegative Matrix Factorization (NMF)}
\label{nmf}

Given a matrix $M \in \mathbb{R}^{m \times n}_+$ and an integer $r \in \mathbb{N}_0$, the \emph{NMF optimization
problem} using the Frobenius norm is defined as
\begin{equation}\label{NMF} \tag{NMF}
\min_{V \in \mathbb{R}^{m \times r}, W \in \mathbb{R}^{r \times n}} ||M-VW||_F^2 \;\;
= \;\;  \sum_{i,j} (M - VW)_{ij}^2 \quad \text{such that } V, W \geq \mathbf{0}
\end{equation}
$\mathbb{R}^{m \times n}$ denotes the set of real matrices of dimension $m \times n$; $\mathbb{R}^{m \times n}_+$ the set of nonnegative matrices i.e. $\mathbb{R}^{m \times n}$ with every entry nonnegative, and $\mathbf{0}$ the zero matrix of appropriate dimensions.\\
A wide range of algorithms have been proposed to find approximate solutions for this problem
(see e.g.~\cite{Ber, Big, Cic, Cic1, Dhi2, Dhi, Gil, Lin}).
Most of them use the fact that although problem~\eqref{NMF} is not convex, its objective function is convex
separately in each of the two factors $V$ and $W$ (which implies that finding the optimal factor $V$ corresponding
to a fixed factor $W$ reduces to a convex optimization problem, and vice-versa), and try to find good approximate solutions by using alternating
minimization schemes. For instance, Nonnegative Least Squares (NNLS) algorithms
can be used to minimize (exactly) the cost function alternatively over factors $V$ and $W$ (see e.g.~\cite{Chen, Kim2}).

Actually, there exist other partitions of the variables that preserve convexity of the alternating minimization
subproblems: since the cost function can be rewritten as $||M- \sum_{i=1}^{r} V_{:i}W_{i:}||_F$, it is clearly convex
as long as variables do not include simultaneously an element of a column of $V$ and an element of the corresponding row of $W$ (i.e.\@ $V_{ki}$ and $W_{il}$ for the same index $i$).
Therefore, given a subset of indexes $K \subseteq R = \{1,2,\dots,r\}$, \eqref{NMF} is clearly convex for both the following subsets of variables
\[ P_K = \Big\{ V_{:i} \; \Big| \; i \in K \Big\} \; \cup \; \Big\{ W_{j:} \; \Big| \; j \in R\setminus K\; \Big\} \]
and its complement
\[ Q_K = \Big\{ V_{:i} \; \Big| \; i \in R\setminus K \Big\} \; \cup \; \Big\{ W_{j:} \; \Big|
\; j \in K\; \Big\} \;. \]

However, the convexity is lost as soon as one column of $V$ ($V_{:i}$) and the corresponding row of
$W$ ($W_{i:}$) are optimized simultaneously, so that the corresponding minimization subproblem can no longer be
efficiently solved up to global optimality.

\subsection{Complexity}
\label{complexityNMF}

Vavasis studies in \cite{Vav} the algorithmic complexity of the NMF optimization problem; more specifically, he proves
that the following problem, called \emph{Exact Nonnegative Matrix Factorization}\footnote{This is closely related to the \textit{nonnegative rank} of matrix $M$,
which is the minimum value of $r$ for which there exists $V \in \mathbb{R}^{m \times r}_+$ and $W \in \mathbb{R}^{r \times n}_+$ such that $M = VW$ (see \cite{Berman}).}, is NP-hard: 

\begin{quote} (Exact NMF) Given a nonnegative matrix $M \geq 0$ of rank $k$, find, if possible, two nonnegative factors $V\geq0$ and $W\geq0$ of rank $k$ such that $M = VW$. \end{quote} \vspace{0.1cm}

\noindent The NMF optimization problem is therefore also NP-hard, since when the rank $r$ is equal to the rank $k$ of the
matrix $M$, any optimal solution to the NMF optimization problem can be used to answer the Exact NMF problem
(the answer being positive if and only if the optimal objective value of the NMF optimization
problem is equal to zero).

The NP-hardness proof for exact NMF relies on its equivalence with a NP-hard problem in polyhedral combinatorics, and
requires both the dimensions of matrix $M$ and its rank $k$ to increase to obtain NP-hardness. In contrast, in the special
cases when rank $k$ is equal to $1$ or $2$, the exact NMF problem can always be answered in the affirmative:
\begin{enumerate}
\item When $k=1$, it is obvious that for any nonnegative rank-one matrix $M \geq 0$ there is nonnegative factors $v \geq 0$ and $w \geq 0$ such that $M = v w^T$.

    Moreover, the NMF optimization problem with $r=1$ can be solved in polynomial time: the Perron-Frobenius theorem implies
    that the dominant left and right singular vectors of a nonnegative matrix $M$ are nonnegative, while the
    Eckart-Young theorem states that the outer product of these dominant singular vectors is the best rank-one
    approximation of $M$; these vectors can be computed in polynomial-time using for example the singular value
    decomposition \cite{Gol}.
\item When nonnegative matrix $M$ has rank $2$, Thomas has shown \cite{Tho} that exact NMF
    is also always possible (see also~\cite{Coh}). The fact that any rank-two nonnegative matrix can be exactly factorized
    as the product of two rank-two nonnegative matrices can be explained geometrically as follows: viewing columns of $M$
    as points in $\mathbb{R}^m$, the fact that $M$ has rank $2$ implies that the set of its columns belongs to a two-dimensional
    subspace. Furthermore, because these columns are nonnegative, they belong to a two-dimensional pointed cone, see Figure~\ref{rank2}.
    Since such a cone is always spanned by two extremes vectors, this implies that all columns of $M$ can be represented exactly as nonnegative
    linear combinations of two nonnegative vectors, and therefore the exact NMF is always possible\footnote{The reason why this property no longer holds for higher values of the rank $k$ is that a $k$-dimensional cone is not necessarily spanned by a set of $k$ vectors when $k > 2$.}.
\begin{figure}[ht]
\begin{center}
\includegraphics[width=8cm]{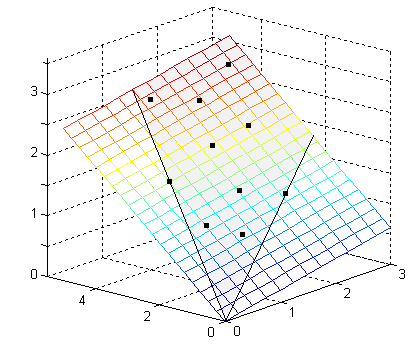}
\caption{Rank-two exact NMF ($k=2$): $m = 3$ and $n = 10$.}
\label{rank2}
\end{center}
\end{figure}

    Moreover, these two extreme columns can easily be computed in polynomial time (using for example the fact that they define an angle of maximum amplitude among all pairs of columns). Hence, when the optimal rank-two approximation of matrix $M$ is nonnegative, the NMF optimization problem
    with $r=2$ can be solved in polynomial time. However, this optimal rank-two approximation is not always nonnegative, 
    so that the complexity of the NMF optimization in the case $r=2$ is not known. Furthermore, to the best of our knowledge, the complexity
    of the exact NMF problem and the NMF optimization problem are still unknown for any fixed rank $r$ or $k$ greater than~$3$.

\end{enumerate}

\subsection{Multiplicative Updates (MU)}
\label{LSalgo}

In their seminal paper \cite{LS2}, Lee and Seung propose multiplicative update rules that aim at minimizing the Frobenius norm between $M$ and $VW$.
To understand the origin of these rules, consider the Karush-Kuhn-Tucker first-order optimality conditions for \eqref{NMF}
\begin{eqnarray}
V \geq \textbf{0},&\quad& W \geq \textbf{0} \label{Vpos} \\
\nabla_V ||M-VW||_F^2 \geq \textbf{0}, &\quad& \nabla_W ||M-VW||_F^2 \geq \textbf{0} \label{Gpos} \\
V \circ \nabla_V ||M-VW||_F^2 = \textbf{0} ,&\quad& W \circ \nabla_W ||M-VW||_F^2 = \textbf{0} \label{Mix}
\end{eqnarray}
where $\circ$ is the Hadamard (component-wise) product between two matrices, and
\begin{equation}
\nabla_V  ||M-VW||_F^2 = -2 (M-VW)W^T,\quad \nabla_W ||M-VW||_F^2 = -2 V^T(M-VW) \;.
\label{GradF}
\end{equation}
Injecting (\ref{GradF}) in (\ref{Mix}), we obtain
\begin{eqnarray}
V \circ (VWW^T) & = & V \circ (MW^T) \label{MixV} \label{LSorigin1}\\
W \circ (V^TVW) & = & W \circ (V^TM) \label{MixW} \;. \label{LSorigin2}
\end{eqnarray}
From these equalities, Lee and Seung derive the following simple multiplicative update rules
(where $\frac{[.]}{[.]}$ is Hadamard (component-wise) division)
\begin{equation}
V \leftarrow V \circ \frac{[M W^T]}{[V W W^T]}, \qquad  W \leftarrow W \circ \frac{[V^T M]}{[V^T V W]}\,
\label{LSupdate}
\end{equation}
for which they are able to prove a monotonicity property:
\begin{theorem}[\cite{LS2}] The Frobenius norm $||M-VW||_F$ is nonincreasing under the multiplicative update rules~\eqref{LSupdate}.
\end{theorem}

\noindent The algorithm based on the alternated application of these rules is not guaranteed to converge to a first-order stationary point (see e.g.~\cite{Ber}, and references therein), but a slight modification proposed in~\cite{Lin2} achieves this property (roughly speaking, MU is recast as a variable metric steepest descent method and the step length is modified accordingly). We propose another possibility to overcome this problem by replacing the above updates by the following:
\begin{theorem} \label{epsMU} For every constant $\epsilon > 0$, $||M-VW||_F$ is nonincreasing under
\begin{equation}
V \leftarrow \max\Big(\epsilon,V \circ \frac{[M W^T]}{[V W W^T]}\Big), \qquad
W \leftarrow \max\Big(\epsilon,W \circ \frac{[V^T M]}{[V^T V W]} \Big)\,
\label{LSupdateEps}
\end{equation}
for any $(V, W) \geq \epsilon$. Moreover, every limit point of this algorithm is a stationary point of the following optimization problem
\begin{equation} \label{epsNMF}
\min_{V\geq \epsilon, W\geq \epsilon} ||M-VW||_F^2.
\end{equation}
\end{theorem}
\begin{proof}
See Section~\ref{MUNF} of this paper, where the more general Theorem~\ref{NFmultEps} is proved.
\end{proof} \vspace{0.1cm}

\subsection{Hierarchical Alternating Least Squares (HALS)} \label{secHALS}
Cichoki et al. \cite{Cic} and independently several other authors \cite{Gil1, ho} have proposed to solve the problem of Nonnegative Matrix Factorization by considering successively each rank-one factor $V_{:k}W_{k:}$ while keeping the rest of the variables fixed, which can be expressed as
\begin{equation}
M \approx V_{:k}W_{k:} + \sum_{i\neq k}V_{:i}W_{i:} \quad \Leftrightarrow \quad V_{:k}W_{k:} \approx M - \sum_{i\neq k}V_{:i}W_{i:} \quad \text{ or } \quad V_{:k}W_{k:} \approx R_k
\end{equation}
where matrix $R_k$ is called the $k^\text{th}$ residual matrix.

Ideally, one would like to find an optimal rank-one factor $V_{:k}W_{k:}$ according to the Frobenius norm, i.e.\@ solve the following problem
\begin{equation}
\min_{V_{:k} \in \mathbb{R}^{m}, W_{k:} \in \mathbb{R}^{n}} ||M-VW||_F^2 = ||R_k - V_{:k} W_{k:} ||_F^2 \quad \text{such that } V_{:k}, W_{k:} \geq \mathbf{0}
\label{symmetry}
\end{equation}
but, instead of solving this problem directly, these authors propose to optimize column $V_{:k}$ and row $W_{k:}$ separately in an alternating scheme, because
the optimal solution to these two (convex) subproblems can be easily computed in closed form, see e.g.~\cite{diep}:
\begin{eqnarray}
V_{:k}^* & = \;\; \textrm{argmin}_{V_{:k} \geq 0} ||R_k - V_{:k} W_{k:} ||_F^2 & = \;\; \max\bigg( \mathbf{0}, \frac{R_k W_{k:}^T}{||W_{k:}||_2^2}\bigg)
\label{gilVcol1} \\
W_{k:}^* & =  \;\; \textrm{argmin}_{W_{k:} \geq 0} ||R_k - V_{:k} W_{k:} ||_F^2 & =  \;\; \max\bigg( \mathbf{0}, \frac{V_{:k}^T R_k}{||V_{:k}||_2^2}\bigg)\; .
\label{gilVcol}
\end{eqnarray}
This scheme, which amounts to a block coordinate descent method (for which any cyclic order on the columns of $V$ and the rows of $W$ is admissible), is called \textit{Hierarchical Alternating Least Squares} (HALS)\footnote{In \cite{diep, ho}, it is called Rank-one Residue Iteration (RRI) method and in \cite{Gil1} Alternating NMF (ANMF).} and it has been observed to work remarkably well in practice: it outperforms, in most cases, the other algorithms for NMF \cite{Cic1, Gil1, diep}. Indeed, it combines a low computational cost per iteration (the same as the multiplicative updates) with a relatively fast convergence (significantly faster than the multiplicative updates), see Figure~\ref{cbcl} for an example.
 We will explain later in Section~\ref{interMU} why this algorithm performs much better than the one of Lee and Seung.

A potential issue with this method is that, in the course of the optimization process, one of the vectors $V_{:k}$ (or $W_{k:}$) and the corresponding
rank-one factor $V_{:k} W_{k:}$ may become equal to zero (this happens for example if one of the residuals $R_k$ is nonpositive). This then leads to numerical instabilities (the next update is not well-defined) and a rank-deficient approximation (with a rank lower than $r$).
A possible way to overcome this problem is to replace the zero lower bounds on $V_{:k}$ and $W_{k:}$ in \eqref{gilVcol1} and \eqref{gilVcol} by a small positive constant, say $\epsilon \lll 1$ (as for the MU), and consider the following subproblems
\begin{equation}
V_{:k}^* = \textrm{argmin}_{V_{:k} \geq \mathbf{\epsilon}} ||R_k - V_{:k} W_{k:} ||_F^2
\text{ and } W_{k:}^*  = \textrm{argmin}_{W_{k:} \geq \mathbf{\epsilon}} ||R_k - V_{:k} W_{k:} ||_F^2, \label{modprobs}
\end{equation}
which lead to the modified closed-form update rules:
\begin{equation}
V_{:k}^* = \max\bigg( \mathbf{\epsilon}, \frac{R_k W^T}{||W_{k:}||_2^2}\bigg)
\quad\text{ and }\quad W_{k:}^*  =  \max\bigg( \mathbf{\epsilon}, \frac{V^T R_k}{||V_{:k}||_2^2}\bigg)\;. \label{modrules}
\end{equation}
This idea was already suggested in \cite{Cic} in order to avoid numerical instabilities.
In fact, this variant of the algorithm is now well-defined in all situations because \eqref{modprobs} guarantees $V_{:k} > 0$ and $W_{k:} > 0$ at each iteration.
Furthermore, one can now easily prove that it converges to a stationary point.
\begin{theorem}
\label{cdTh}
For every constant $\epsilon > 0$, the limit points of the block coordinate descent algorithm initialized with positive matrices and applied to the optimization problem \eqref{epsNMF} are stationary points.
\end{theorem}
\begin{proof}
We use the following result of Powell \cite{Pow} (see also~\cite[p.268]{Bert}): the limit points of the iterates of a block coordinate descent algorithm are stationary points provided that the following two conditions hold:
\begin{itemize}
\item each block of variables is required to belong to a closed convex set,
\item the minimum computed at each iteration for a given block of variables is uniquely attained.
\end{itemize}
The first condition is clearly satisfied here, since $V_{:k}$ and $W_{k:}$ belong respectively to $([\epsilon,+\infty[)^m$ and $([\epsilon,+\infty[)^n$, 
which are closed convex sets. The second condition holds because subproblems~\eqref{modprobs} can be shown to be strictly convex, so that their optimal value is uniquely attained by the solutions provided by rules~\eqref{modrules}. Strict convexity is due to the fact that the objective function of these problems
are sums of quadratic terms, each involving a single variable and having a strictly positive coefficient.
\end{proof}
\vspace{0.2cm}

\section{Nonnegative Factorization (NF)}\label{SecNF}

Looking back at subproblem~\eqref{symmetry}, i.e.\@ approximating the residual $R_k$ with a rank-one term $V_{:k} W_{k:}$, we have seen that the optimal solution separately for both $V_{:k}$ and $W_{k:}$ can be written in a closed form. In the previous section, subproblem~\eqref{symmetry} was then solved by a block coordinate descent algorithm.

A question arises: Is it possible to do better? i.e.\@ \textit{Is it possible to efficiently solve the problem for both vectors simultaneously?}
In order to answer this question, we introduce the problem of Nonnegative Factorization\footnote{This terminology has already been used for the
problem of finding a symmetric nonnegative factorization, i.e.\@ one where V=W, but we assign it a different meaning in this paper.
} which is exactly the same as Nonnegative Matrix Factorization except that the matrix to factorize can be \textit{any real matrix}, i.e.\@ is not necessarily nonnegative.
Given $M \in \mathbb{R}^{m \times n}$ and $r \in \mathbb{N}_0$, the Nonnegative Factorization optimization problem using the Frobenius norm is:
\begin{align}
\min_{V \in \mathbb{R}^{m \times r}, W \in \mathbb{R}^{r \times n}} & ||M-VW||_F^2 \nonumber \\
&  V\geq \mathbf{0},\;\; W \geq \mathbf{0}  \label{NF}  \tag{NF}
\end{align}
Of course, this problem is a generalization of \eqref{NMF} and is NP-hard as well. However Nonnegative Factorization will be shown below to be NP-hard for any \emph{fixed} factorization rank (even $r=1$), which is not the case of \eqref{NMF} (cf. Section~\ref{complexityNMF}). The proof is based on the reduction to the maximum edge biclique problem.

\subsection{Complexity}
\label{complex}

The main result of this section is the NP-hardness result for Nonnegative Factorization for any fixed factorization rank. We first show how the optimization version of the maximum edge biclique problem \eqref{MBP} can be formulated as a rank-one Nonnegative Factorization problem \eqref{NF1}. Since the decision version of \eqref{MBP} is NP-complete \cite{Peet}, this implies that \eqref{NF1} is NP-hard. We then prove that \eqref{NF} is NP-hard as well using a simple construction.

\subsubsection*{The Maximum Edge Biclique Problem in Bipartite Graphs} \label{aboutbic}

A \emph{bipartite graph} $G_b$ is a graph whose vertices can be divided into two disjoint sets $V_1$ and $V_2$ such
that there is no edge between two vertices in the same set
\begin{displaymath}
G_b = (V,E) = \Big(V_1 \cup V_2, E \subseteq (V_1 \times V_2) \Big).
\end{displaymath}
A \emph{biclique} $K_b$ is a complete bipartite graph i.e.\@ a bipartite graph where all the vertices are connected
\begin{displaymath}
K_b = (V',E') = \Big(V'_1 \cup V'_2, E' = (V'_1 \times V'_2) \Big).
\end{displaymath}

\noindent Finally, the so-called maximum edge biclique problem in a bipartite graph~$G_b = (V,E)$ is the problem of finding a biclique~$K_b = (V',E')$ in $G_b$ (i.e.\@ $V' \subseteq V$ and $E' \subseteq E$) maximizing the number of edges. The decision problem: \textit{Given $B$, does $G_b$ contain a biclique with at least $B$ edges?} has been shown to be NP-complete \cite{Peet}. Therefore the corresponding optimization problem is at least NP-hard.\\

Let $M_b \in \{ 0,1 \}^{m \times n}$ be the adjacency matrix of the unweighted bipartite graph $G_b = (V_1 \cup V_2,E)$ 
i.e.\@ $M_b(i,j) = 1$ if and only if $(V_1(i),V_2(j)) \in E$.
In order to avoid trivialities, we will suppose that each vertex of the graph is connected to at least one other vertex i.e.~\@ $M_b(i,:) \neq \mathbf{0}, M_b(:,j) \neq \mathbf{0}, \forall i,j$.
We denote by $|E|$ the cardinality of $E$ i.e.\@ the number of edges in $G_b$; note that~$|E| = ||M_b||_F^2$. The set of zero values will be called $Z = \{(i,j) \, | \, M_b(i,j) = 0\}$, and its cardinality $|Z|$ satisfies $|E| + |Z| = mn$.\\

\noindent With this notation, the maximum biclique problem in $G_b$ can be formulated as
\begin{align}
\qquad \qquad \min_{v, w} \qquad & ||M_b-vw||_F^2 \nonumber \\
&  v \in \{ 0,1 \}^{m}, w \in \{ 0,1 \}^{n} \label{MBP}  \tag{MBP} \\
&  v_i w_j \leq M_b(i,j), \, \forall i,j \nonumber
\end{align}
In fact, one can check easily that this objective is equivalent to $\max_{v,w} \sum_{ij} v_iw_j$ since $M_b$, $v$ and $w$ are binary: instead of maximizing the number of edges inside the biclique, one minimizes the number of edges outside. \\
Feasible solutions of \eqref{MBP} correspond to bicliques of $G_b$.
We will be particularly interested in \emph{maximal} bicliques. A maximal biclique is a biclique which is not contained in a larger biclique: it is a locally optimal solution of \eqref{MBP}.\\
\noindent The corresponding rank-one Nonnegative Factorization problem is defined as
\begin{align}
\quad \; \min_{v \in \mathbb{R}^{m}, w \in \mathbb{R}^{n}} & ||M_d-vw||_F^2 \nonumber \\
&  v \geq \mathbf{0},\; w \geq \mathbf{0} \label{NF1} \tag{NF-1d}
\end{align}
with the matrix $M_d$ defined as
\begin{equation}
\label{Md}
M_d = (1+d)M_b - d \, \mathbf{1}_{m \times n}, \quad d > 0
\end{equation}
where $\mathbf{1}_{m \times n}$ is the matrix of all ones with dimension $m \times n$. $M_d$ is the matrix $M_b$ where the zero values have been replaced by $-d$. Clearly $M_d$ is not necessarily a nonnegative matrix.\\

\noindent To prove NP-hardness of \eqref{NF1}, we are going to show that, if $d$ is sufficiently large, optimal solutions of \eqref{NF1} \textit{coincide} with optimal solutions of the corresponding biclique problem \eqref{MBP}.
From now on, we say that a solution $(v,w)$ coincides with another solution $(v',w')$  if and only if $vw = v'w'$ (i.e.\@ if and only if $v' = \lambda v$ and $w' = \lambda^{-1} w$ for some $\lambda > 0$). We also let $M_+ = \max(\mathbf{0},M)$ and $M_- = \max(\mathbf{0},-M)$. \vspace{0.1cm}

\begin{lemma}
\label{lem1}
Any optimal rank-one approximation with respect to the Frobenius norm of a matrix $M$ for which $\min(M) \leq -||M_+||_F$ contains at least one nonpositive entry.
\end{lemma}
\begin{proof}
If $M = \mathbf{0}$, the result is trivial.
If not, $\min(M) < 0$ since $\min(M) \leq -||M_+||_F$. Suppose now $(v,w) > 0$ is a best rank-one approximation of $M$. Therefore, since the negative values of $M$ are approximated by positive ones and since $M$ has at least one negative entry, we have
\begin{equation} \label{lw}
||M-vw||_F^2 > ||M_-||_F^2.
\end{equation}
By the Eckart-Young theorem,
\begin{displaymath}
||M-vw||_F^2 = ||M||_F^2 - \sigma_{max}(M)^2 = ||M||_F^2 -||M||_2^2,
\end{displaymath}
where  $\sigma_{max}(M)$ is the maximum singular value of $M$. Clearly,
\[ ||M||_F^2 = ||M_+||_F^2 + ||M_-||_F^2 \quad \textrm{ and } \quad ||M||_2^2 \geq \min(M)^2. \]
So
\begin{displaymath}
||M-vw||_F^2 \leq ||M_+||_F^2 + ||M_-||_F^2 - \min(M)^2 \leq ||M_-||_F
\end{displaymath}
which is in contradiction with \eqref{lw}.
\end{proof} \vspace{0.2cm}

We restate here a well-known result concerning low-rank approximations (see e.g.~\cite[p. 29]{diep}).
\begin{lemma}
\label{lem2}
The local minima of the best rank-one approximation problem with respect to the Frobenius norm are global minima.
\end{lemma} \vspace{0.2cm}

We can now state the main result about the equivalence of \eqref{NF1} and \eqref{MBP}.
\begin{theorem}
\label{thp}
For $d \geq \sqrt{|E|}$, any optimal solution (v,w) of \eqref{NF1} coincides with an optimal solution of \eqref{MBP}, i.e.\@ $vw$ is binary and $vw \leq M_b$.
\end{theorem}
\begin{proof}
We focus on the entries of $vw$ which are positive and define
\begin{equation} \label{ST}
K = \Big\{i \in \{1,2,\dots,m\} \;\Big|\; v_i > 0\Big\} \quad \textrm{ and } \quad
L = \Big\{j \in \{1,2,\dots,n\} \;\Big|\; w_j > 0\Big\}.
\end{equation}
$v' = v(K)$, $w' = w(L)$ and $M_d' = M_d(K,L)$ are the submatrices with indexes in $(K,L)$.
Since $(v,w)$ is optimal for $M_d$, $(v',w')$ must be optimal for $M_d'$.
Suppose there is a $-d$ entry in $M_d'$, then
\[ \min(M_d') = -d \leq -\sqrt{|E|} = -||(M_d)_+||_F \leq -||(M_d')_+||_F,\]
so that Lemma \ref{lem1} holds for $M_d'$. Since $(v',w')$ is positive and is an optimal solution of \eqref{NF1} for $M_d'$, $(v',w')$ is a local minimum of the unconstrained problem i.e.\@ the problem of best rank-one approximation. By Lemma \ref{lem2}, this must be a global minimum. This is a contradiction with Lemma~\ref{lem1}: $(v',w')$ should contain at least one nonpositive entry.
Therefore, $M'_d = \mathbf{1}_{|K| \times |L|}$ which implies $v'w' = M'_d$ by optimality and then $vw$ is binary and $vw \leq M_b$.
\end{proof} \vspace{0.2cm}

\begin{corollary}
Rank-one Nonnegative Factorization is NP-hard.
\label{NF1np}
\end{corollary} \vspace{0.2cm}

Intuitively, the reason \eqref{NF1} is NP-hard is that if one of the $-d$ entries of $M_d$ is approximated by a positive value, say $p$, the corresponding error is 
$d^2 + \mathbf{2pd} + p^2$. Therefore, the larger $d$, the more expensive it is to approximate $-d$ by a positive number. Because of that, when $d$ increases, negatives values of $M_d$ will be approximated by smaller values and eventually by zeros.

Therefore, for each negative entry of $M$, one has to decide whether to approximate it with zero or with a positive value. Moreover, when a value is approximated by zero, one has to choose which entries of $V$ and $W$ will be equal to zero, as in the biclique problem. We suspect that the hardness of Nonnegative Factorization lies in these combinatorial choices.\\

\noindent We can now answer our initial question: Would it be possible to solve efficiently the problem
\begin{displaymath}
V_{:k}W_{k:} \approx R_k = M - \sum_{i\neq k}V_{:i}W_{i:} \ngeq \mathbf{0}
\end{displaymath}
simultaneously for both vectors $(V_{:k},W_{k:})$?
Obviously, unless P=NP, we won't be able to find a polynomial-time algorithm to solve this problem.
Therefore, it seems hopeless to improve the HALS algorithm using this approach.\\

\begin{remark}
Corollary \ref{NF1np} suggests that NMF is a difficult problem for any fixed $r\geq2$. Indeed, even if one was given the optimal solution of an NMF problem except for one rank-one factor, it is not guaranteed that one would be able to find this last factor in polynomial-time since the corresponding residue is not necessarily nonnegative.
\end{remark} \vspace{0.2cm}

We now generalize Theorem~\ref{NF1np} to factorizations of arbitrary rank.
\begin{theorem}
\label{nfNP}
Nonnegative Factorization \eqref{NF} is NP-hard.
\end{theorem}
\begin{proof}
Let $M_b \in \{0,1\}^{m \times n}$ be the adjacency matrix of a bipartite graph $G_b$ and $r \geq 1$ the factorization rank of \eqref{NF}.
We define the matrix $A_b$ as
\begin{displaymath}
A_b = \textrm{diag}(M_b,r) = \left( \begin{array}{cccc}   M_b & \mathbf{0} & \dots & \mathbf{0} \\
\mathbf{0} & M_b & & \mathbf{0} \\
\vdots & & \ddots & \vdots \\
\mathbf{0} & \dots & & M_b
 \end{array} \right)
\end{displaymath}
which is the adjacency matrix of another bipartite graph $G_b'$ which is nothing but the graph $G_b$ repeated $r$
times. $A_d$ is defined in the same way as $M_d$, i.e.
\[ A_d = (1+d)A_b - d \, \mathbf{1}_{m \times n}\]
with $d \geq  \sqrt{r|E|}$. Let $(V,W)$ be the optimal rank-$r$ nonnegative factorization of $A_d$ and consider each rank-one factor $V_{:k}W_{k:} \approx R_k = A_d - \sum_{i\neq k} V_{:i}W_{i:}$: each of them must clearly be an optimal rank-one nonnegative factorization of $R_k$.
Since $R_k \leq A_d$,
\[ \min(R_k) \leq \min(A_d) = -d \leq -||(A_d)_+||_F \leq -||(R_k)_+||_F,\]
and Lemma \ref{lem1} holds. Using exactly the same arguments as in Theorem~\ref{thp}, one can show that, $\forall k$,
\begin{displaymath}
(V_{:k}W_{k:})_{ij} = 0, \quad \forall \; (i,j) \; \textrm{ s.t. } \; A_d(i,j) = -d.
\end{displaymath}
Therefore, the positive entries of each rank-one factor will correspond to a biclique of $G_b'$.
By optimality of $(V,W)$, each rank-one factor must correspond to a maximum biclique of $G_b$ since $G_b'$ is the graph $G_b$ repeated $r$ times. Thus \eqref{NF1} is NP-hard implies \eqref{NF} is NP-hard.
\end{proof} \vspace{0.2cm}

\subsection{Stationary points of \eqref{NF1}}

We have shown that optimal solutions of \eqref{NF1} coincide with optimal solutions of \eqref{MBP} for $d \geq \sqrt{E}$, which are NP-hard to find.
In this section, we focus on stationary points of \eqref{NF1} instead: we show how they  are related to the feasible solutions of \eqref{MBP}. This result will be used in Section \ref{mbfa} to design a new type of biclique finding algorithm.

\subsubsection{Definitions and Notation}

The KKT conditions of \eqref{NF1}, which define the stationary points, are exactly the same as for \eqref{NMF}: $(v,w)$ is a stationary point of \eqref{NF1} if and only if
\begin{align}
v  \geq \mathbf{0}, \quad \mu = (vw-M_d)w^T \quad \geq \mathbf{0}  & \quad \textrm{ and } \quad v \circ \mu = \mathbf{0} \label{statnf11} \\
w \geq \mathbf{0}, \quad \lambda = v^T(vw-M_d) \quad \geq \mathbf{0}  & \quad \textrm{ and } \quad w \circ \lambda = \mathbf{0}.  \label{statnf12}
\end{align}
Of course, we are especially interested in nontrivial solutions and we then assume $v, w \neq \mathbf{0}$ so that one can check that \eqref{statnf11}-\eqref{statnf12} are equivalent to
\begin{equation} \label{statNF1}
v = \max\Big( \mathbf{0}, \frac{M_dw^T}{||w||_2^2} \Big)
\qquad \textrm{ and } \qquad
w = \max\Big( \mathbf{0}, \frac{v^TM_d}{||v||_2^2} \Big).
\end{equation}
Given $d$, we define three sets of rank-one matrices: $S_d$, corresponding to the nontrivial stationary points of \eqref{NF1}, with
\[
S_d = \{ vw \in \mathbb{R}^{m \times n}_{0}  \; | \;
(v,w) \textrm{ satisfy } \eqref{statNF1} \},
\]
$F$, corresponding to the feasible solutions of \eqref{MBP}, with
\[
F = \{ vw \in \mathbb{R}^{m \times n}
\; | \; (v,w) \textrm{ is a  feasible of } \eqref{MBP}  \},
\]
and $B$, corresponding to the maximal bicliques of \eqref{MBP}, i.e.\@ $vw \in B$ if and only if $vw \in F$ and $vw$ coincides with a maximal biclique.

\subsubsection{Stationarity of Maximal Bicliques}

The next theorem states that, for $d$ sufficiently large, the only nontrivial feasible solutions of \eqref{MBP} that are stationary points of \eqref{NF1} are the maximal bicliques.
\begin{theorem}
\label{th3v}
For $d > \max(m,n)-1$,\; $F \cap S_d = B$.
\end{theorem}
\begin{proof}
$vw \in B$ if and only if $vw \in F$ and is maximal i.e.
\begin{itemize}
\item[(1)] $\nexists i  \textrm{ such that }  v_i = 0
\textrm{ and }  M_d(i,j) = 1, \forall j \textrm{ s.t. } w_j \neq 0,$
\item[(2)] $\nexists j  \textrm{ such that }  w_j = 0
\textrm{ and }  M_d(i,j) = 1, \forall i \textrm{ s.t. } v_i \neq 0.$
\end{itemize}
Since $vw$ is binary and $v, w\neq \mathbf{0}$, the nonzero entries of $v$ and $w$ must be equal to each other. Moreover, $d > \max(m,n)-1$ so that (1) is equivalent to
\vspace{0.1cm}
\begin{center}
$\quad \nexists \; i \quad \textrm{ such that } \quad v_i = 0 \quad  \textrm{ and } \quad M_d(i,:)w^T > 0$  \\
$\iff$   \\
$v_i = 0 \; \Rightarrow \; M_d(i,:)w^T < 0
\quad \textrm{ and } \quad v_i \neq 0 \;  \Rightarrow \; v_i = \frac{||M_d(i,:)||_1}{||w||_1} = \frac{M_d(i,:)w^T}{||w||_2^2}$.
\end{center}
This is equivalent to the stationarity conditions for $v \neq \mathbf{0}$, cf. \eqref{statNF1}.
By symmetry, (2) is equivalent to the stationarity conditions for $w$.
\end{proof} \vspace{0.2cm}

Theorem~\ref{th3v} implies that, for $d$ sufficiently large, $B \subset S_d$.
It would be interesting to have the opposite affirmation: for $d$ sufficiently large, any stationary point of \eqref{NF1} corresponds to a maximal biclique of \eqref{MBP}.
Unfortunately, we will see later that this property does not hold.

\subsubsection{Limit points of $S_d$}

However, as $d$ goes to infinity, we are going to show that the points in $S_d$ get closer to feasible solutions of \eqref{MBP}. \vspace{0.1cm}

\begin{lemma} \label{Sdbounded} The set $S_d$ is bounded i.e.\@
$\forall d > 0$, $\forall vw \in S_d$:
\[ ||vw||_2 = ||v||_2||w||_2 \leq \sqrt{|E|}. \]
\end{lemma}
\begin{proof}
For $vw \in S_d$, by \eqref{statNF1},
\begin{displaymath}
||v||_2  =    \Big|\Big| \max \Big( \mathbf{0}, \frac{M_dw^T}{||w||_2^2}\Big)  \Big|\Big|_2
        \leq  \frac{|| \max ( \mathbf{0},M_d ) w^T ||_2 }{||w||_2^2}
				\leq  \frac{|| \max ( \mathbf{0},M_d) ||_F}{|| w ||_2}
        =  \frac{\sqrt{|E|}}{||w||_2}.
\end{displaymath}
\end{proof} 

\begin{lemma} \label{bici}
For $vw \in S_d$, if $M_d(i,j) = -d$ and if $(vw)_{ij} > 0$, then
\[ 0 < v_i < \frac{||v||_1}{d+1}
\quad \textrm{ and } \quad  0 < w_j < \frac{||w||_1}{d+1} .\]
\end{lemma}
\begin{proof} By \eqref{statNF1},
\[0 <  w_j ||v||_2^2 = {v^TM_d(:,j)}{} \leq {||v||_1 - (d+1) v_i}{}
\; \Rightarrow \; 0 < v_i < \frac{||v||_1}{d+1}.\]
The same can be shown for $w$ by symmetry.
\end{proof} \vspace{0.2cm}

\begin{theorem} \label{Sdinf}
As $d$ goes to infinity, stationary points of \eqref{NF1} get closer to feasible solutions of \eqref{MBP} i.e. $\forall \epsilon > 0$, $\exists D$ s.t. $\forall d > D$:
\begin{equation} \label{SdF}
\max_{vw \in S_d} \; \min_{v_bw_b \in F} \; ||vw-v_bw_b||_F \; < \; \epsilon.
\end{equation}
\end{theorem}
\begin{proof}
Let $vw \in S_d$ and suppose $vw > 0$. W.l.o.g. $||w||_2 = 1$; in fact, if $vw \in S_d$, $\Big(\lambda v \frac{1}{\lambda} w\Big) \in S_d, \forall \lambda > 0$.
Note that Lemma \ref{Sdbounded} implies $||v||_2 \leq \sqrt{|E|}$.
By \eqref{statNF1},
\[ v = M_{d}w^T \quad \textrm{ and } \quad
w = \frac{v^TM_{d}}{||v||_2^2}.\]
Therefore, $(v/||v||_2,w) > 0$ is a pair of singular vectors of $M_{d}$ associated with the singular value $||v||_2 > 0$. If $M_{d} = \mathbf{1}_{m \times n}$, the only  pair of positive singular vectors of $M_d$ is  $\Big( \frac{1}{\sqrt{m}}\mathbf{1}_{m},\frac{1}{\sqrt{n}}\mathbf{1}_{n} \Big)$ so that $vw = M_b$ coincides with a feasible solution of \eqref{MBP}.\\
Otherwise, we define
\begin{equation} \label{AB}
A = \Big\{ i  \;\Big|\; M_d(i,j) = 1, \forall j \Big\} \quad \textrm{ and } \quad
B = \Big\{ j  \;\Big|\; M_d(i,j) = 1, \forall i \Big\},
\end{equation}
and their complements $\bar{A} = \{1,2,\dots,m\} \backslash A$, $\bar{B} = \{1,2,\dots,n\} \backslash B$; hence,
\[ M_d(A,:) = \mathbf{1}_{|A| \times n} \quad \textrm{ and } \quad M_d(:,B) = \mathbf{1}_{m \times |B|}.\]
Using Lemma \ref{bici} and the fact that $||x||_1 \leq \sqrt{n}||x||_2, \forall x \in \mathbb{R}^n$, we get
\begin{equation} \label{lembici}
\mathbf{0} < v(\bar{A}) < \frac{\sqrt{m|E|}}{d+1} \, \mathbf{1}_{|\bar{A}|} \quad \textrm{ and } \quad
\mathbf{0} < w(\bar{B}) < \frac{\sqrt{n}}{d+1} \, \mathbf{1}_{|\bar{B}|}.
\end{equation}
Therefore, since $w \leq \mathbf{1}_{n}$ and $v \leq \sqrt{|E|} \, \mathbf{1}_{m}$, we obtain
\[
||v(\bar{A})w - \mathbf{0}||_F < \frac{1}{d+1} \Big( m \sqrt{n|E|} \Big)
\;\; \textrm{ and } \;\;
||v w(\bar{B})-\mathbf{0}||_F < \frac{1}{d+1} \Big( n \sqrt{m|E|} \Big).
\]
It remains to show that $v(A)w(B)$ coincide with a biclique of the (complete) graph generated by $M_b(A,B) = \mathbf{1}_{|A| \times |B|}$. We distinguish three cases:\\
(1) $A = \emptyset$. \eqref{lembici} implies
$||v||_2 < \frac{1}{d+1}\Big(m\sqrt{m|E|}\Big)$ so that
\begin{equation} \label{wB}
 w(B) 
 = \frac{v^T\mathbf{1}_{m \times |B|}}{||v||_2^2}
 = \frac{||v||_1}{||v||_2^2}\,\mathbf{1}_{|B|}
 \geq 
	 \frac{\sqrt{m}}{||v||_2}\,\mathbf{1}_{|B|} >  \frac{d+1}{m\sqrt{|E|}} \,\mathbf{1}_{|B|}
\end{equation}
which is absurd if $d > m\sqrt{|E|}$ since $||w||_2 = 1$.
\\
(2) $B = \emptyset$. Using \eqref{lembici}, we have
$v(A) = M_d(A,:)w^T = ||w||_1 \mathbf{1}_{|A|} < n \frac{\sqrt{n}}{d+1}$ and then
\[
||v(A)w(B)-\mathbf{0}||_F
< \frac{1}{d+1} \Big( n^2 \sqrt{m}   \Big). 
\]
(3) $A, B \neq \emptyset$. Noting $k_w = \frac{||v||_1}{||v||_2^2}$, Equation \eqref{wB} gives $w(B) = k_w \,\mathbf{1}_{|B|}$.
Therefore,
\begin{equation} \label{kw}
 1 - |\bar{B}| \frac{  \sqrt{n}}{d+1}
 < ||w||_2^2 - ||w(\bar{B})||_2^2 = ||w(B)||_2^2 = |B| k_w^2 \leq ||w||_2^2 = 1,
\end{equation}
Moreover, $v(A) =  \mathbf{1}_{|A| \times m}w^T =  ||w||_1 \mathbf{1}_{|A|}$ so that
\begin{equation} \label{kv}
 |B| k_w
	\leq v(A) = (||w(B)||_1 + ||w(\bar{B})||_1) \mathbf{1}_{|A|}
	< |B| k_w + |\bar{B}| \frac{ \sqrt{n}}{d+1}.
	\end{equation}Finally, combining \eqref{kw} and \eqref{kv} and noting that $k_w \leq 1$ since $||w||_2 = 1$,
\[
\Big( 1 - \frac{|\bar{B}| \sqrt{n}}{d+1} \Big) \, \mathbf{1}_{|A| \times |B|}
< v(A)w(B)
<  \Big( 1 + \frac{|\bar{B}| \sqrt{n}}{d+1} \Big) \, \mathbf{1}_{|A| \times |B|}.
\]
We can conclude that, for $d$ sufficiently large, $vw$ is arbitrarily close to a feasible solution of \eqref{MBP} which corresponds to the biclique $(A,B)$.\\

\noindent Recall we supposed $vw > 0$. If $vw \ngtr 0$, let $(K,L)$ be the indexes defined in \eqref{ST}. The above result holds for $v(K)w(L) > 0$ with the matrix $M_d(K,L)$. For $d$ sufficiently large, $v(K)w(L)$ is then close to a feasible solution of \eqref{MBP} for $M_d(K,L)$. Adding zero to this feasible solution gives a feasible solution for $M_b$.
\end{proof}\vspace{0.1cm}

\begin{example} \label{convSd} Let
\begin{displaymath}
M_d = \left( \begin{array}{cc}
-d & 1   \\
1 & 1   \\\end{array} \right).
\end{displaymath}
Clearly, 
$\left( \begin{array}{cc} 0 & 1 \\ 0 & 1 \end{array}  \right)$ belongs to the set B, i.e.\@ it corresponds to maximal biclique of the graph generated by $M_b$.
By Theorem $\ref{th3v}$, for $d > 1$, it belongs to $S_d$ i.e. $[(0\;1)^T,(0\;1)]$ is stationary points of \eqref{NF1}.\\
For $d > 1$, one can also check that the singular values of $M_d$ are disjoint and that the second pair of singular vectors is positive. Since it is a positive stationary point of the unconstrained problem, it is also a stationary point of \eqref{NF1}.
As $d$ goes to infinity, it must get closer to a biclique of \eqref{MBP} (Theorem~\ref{Sdinf}). Moreover $M_d$ is symmetric so that the right and left singular vectors are equal to each other. Figure \ref{contSd} shows the evolution\footnote{By Wedin's theorem (cf. matrix perturbation theory, see e.g.~\cite{Stew}), singular subspaces of $M_d$ associated with a positive singular value are continuously deformed with respect to $d$.}
of this positive singular vector of $M_d$ with respect to $d$.
It converges to $(0\;1)$ and then the product of the left and right singular vector converges to $\left( \begin{array}{cc} 0 & 0 \\ 0 & 1 \end{array}  \right) \in F$.
\begin{figure}[ht]
\begin{center}
\includegraphics[width=7cm]{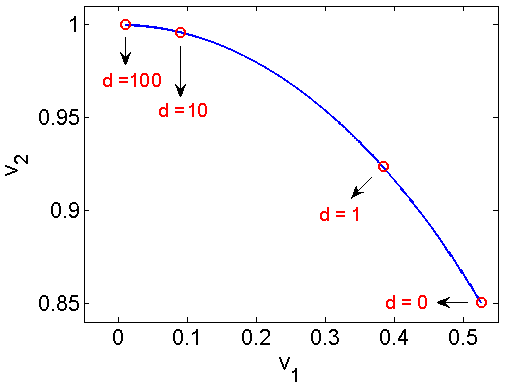}
\caption{Evolution of $(v_1 v_2)^T(v_1 v_2) \, \in \,  S_d$.}
\label{contSd}
\end{center}
\end{figure}
\end{example} \vspace{0.1cm}

\newpage

\subsection{Multiplicative Updates for Nonnegative Factorization}
\label{MUNF}
In this section, the MU of Lee and Seung presented in Section~\ref{LSalgo} to find approximate solutions of \eqref{NMF} are generalized to \eqref{NF}. Other than providing a way of computing approximate solutions of \eqref{NF}, this result will also help us to understand why the updates of Lee and Seung are not very efficient in practice.\\
The Karush-Kuhn-Tucker optimality conditions of the \eqref{NF} problem are the same as for \eqref{NMF} (see Section~\ref{LSalgo}). Of course, any real matrix $M$ can be written as the difference of two nonnegative matrices: $M=P-N$ with $P, N \geq \mathbf{0}$. This can be used to generalize the algorithm of Lee and Seung. In fact, \eqref{LSorigin1} and \eqref{LSorigin2}  become
\begin{eqnarray}
V \circ (VWW^T+ N W^T) & = & V \circ (PW^T) \label{MixVnf}\\
W \circ (V^TVW+ V^T N) & = & W \circ (V^TP) \label{MixWnf}
\end{eqnarray}
and using the same idea as in Section~\ref{LSalgo}, we get the following multiplicative update rules :
\begin{theorem}
\label{NFmultTh}
For $V,W \geq \mathbf{0}$ and $M = P-N$ with $P,N \geq \mathbf{0}$, the cost function
$||M-VW||_F$ is nonincreasing under the following update rules:
\begin{equation}
\label{NFmult}
V \leftarrow V \circ \frac{[P W^T]}{[V W W^T+N W^T]}\,, \qquad  W \leftarrow W \circ \frac{[V^T P]}{[V^T V W+V^T N]}\,.
\end{equation}
\end{theorem}
\begin{proof} We only treat the proof for $V$ since the problem is perfectly symmetric.
The cost function can be split into $m$ independent components related to each row of the error matrix, each depending on a specific row of $P$, $N$ and $V$, which we call respectively $p$, $n$ and $v$. Hence, we can treat each row of V separately, and we only have to show that the function
\[ F(v) = \frac{1}{2} ||p-n-vW||_F^2.\]
is nonincreasing under the following update
\begin{equation} \label{cfG}
v_0 \leftarrow v_0 \circ \frac{[pW^T]}{[v_0WW^T + n W^T]}, \; \forall v_0 > 0.
\end{equation}
$F$ is a quadratic function so that
\[ F(v) = F(v_0) + (v-v_0) \nabla F(v_0) + \frac{1}{2} (v-v_0) \nabla^2 F(v_0) (v-v_0)^T, \quad \forall v_0, \]
with $\nabla F(v_0) = (-p+n+v_0W) W^T$ and $\nabla^2 F(v_0) = WW^T$. Let $G$ be a quadratic model of $F$ around $v_0$:
\[ G(v) = F(v_0) + (v-v_0) \nabla F(v_0) + \frac{1}{2} (v-v_0) K(v_0) (v-v_0)^T \]
with $K(v_0) = \textrm{diag} \Big( \frac{[v_0WW^T + n W^T]}{[v_0]} \Big)$. $G$ has the following nice properties (see below):
\begin{itemize}
\item[(1)] $G$ is an upper approximation of $F$ i.e. $G(v) \geq F(v), \forall v$;
\item[(2)] The global minimum of $G(v)$ is nonnegative and given by \eqref{cfG}. \vspace{0.1cm}
\end{itemize}
Therefore, the global minimum of $G$, given by \eqref{cfG}, provides a new iterate which guarantee the monotonicity of $F$. In fact,
\[ F(v_0) = G(v_0) \geq \min_v G(v) = G(v^*) \geq F(v^*). \]
It remains to show that (1) and (2) hold.\\
\emph{(1)} $G(v) \geq F(v)\; \forall v$. This is equivalent to $K(v_0)-WW^T$ positive semidefinite (PSD). Lee and Seung have proved \cite{LS2} that $A =\textrm{diag} \Big(\frac{[v_0WW^T]}{[v_0]}\Big)-WW^T$ is PSD (see also~\cite{ho2}).
Since $B = \textrm{diag} \Big(\frac{[n W^T]}{[v_0]} \Big)$ is a diagonal nonnegative matrix for $v_0 > 0$ and $nW^T \geq 0$, $A+B = K(v_0)-WW^T$ is also PSD. \vspace{0.02cm}

\noindent \emph{(2)} The global minimum of $G$ is given by \eqref{cfG}:
\begin{eqnarray*}
v^* = \textrm{argmin}_v G(v) & = & v_0 - K^{-1}(v_0) \nabla F(v_0) \\
													   & = & v_0 - v_0 \circ \frac{[-pW^T+(v_0WW^T +n W^T)]}{[v_0WW^T + n W^T]} \\
								  					 & = & v_0 \circ \frac{[pW^T]}{[v_0WW^T + n W^T]}.
\end{eqnarray*}

\end{proof} \vspace{0.1cm}

As with standard multiplicative updates, convergence can be guaranteed with a simple modification:
\begin{theorem} \label{NFmultEps} For every constant $\epsilon > 0$ and for $M = P-N$ with $P,N \geq 0$, $||M-VW||_F$ is nonincreasing under
\begin{equation} \label{NFupdateEps}
V \leftarrow \max\Big(\epsilon,V \circ \frac{[P W^T]}{[V W W^T + NW^T]}\Big), \;
W \leftarrow \max\Big(\epsilon,W \circ \frac{[V^T P]}{[V^T V W + V^TN]} \Big)\,
\end{equation}
for any $(V, W) \geq \epsilon$. Moreover, every limit point of this algorithm is a stationary point of the optimization problem \eqref{epsNMF}.
\end{theorem}
\begin{proof}
We use exactly the same notation as in the proof of Theorem~\ref{NFmult}, so that
\[
F(v_0) = G(v_0) \geq \min_{v\geq \epsilon} G(v)  = G(v^*) \geq F(v^*), \quad v_0 \geq \epsilon
\]
remains valid. By definition, $K(v_0)$ is a diagonal matrix implying that \emph{$G(v)$ is the sum of $r$ independent quadratic terms}, each depending on a single entry of $v$. Therefore,
\[
\textrm{argmin}_{v\geq \epsilon} G(v) = \max\Big(\epsilon,v_0 \circ \frac{[pW^T]}{[v_0WW^T+nW^T]}\Big),
\]
and the monotonicity is proved.\\
Let $(\bar{V},\bar{W})$ be a limit point of a sequence $\{(V^k,W^k)\}$ generated by \eqref{NFupdateEps}.
The monotonicity implies that $\{||M-V^kW^k||_F\}$ converges to $||M-\bar{V}\bar{W}||_F$ since the cost function is bounded from below. Moreover,
\begin{equation} \label{limpoint}
 \bar{V}_{ik} = \max\Big(\epsilon,\alpha_{ik} \; \bar{V}_{ik} \Big), \quad \forall i,k
 \end{equation}
where
\[ \alpha_{ik} = \frac{P_{i:}\bar{W}_{k:}^T}{\bar{V}_{i:} \bar{W} \bar{W}_{k:}^T + N_{i:}\bar{W}_{k:}^T},\]
which is well-defined since $\bar{V}_{i:} \bar{W} \bar{W}_{k:}^T > 0$. One can easily check that the stationarity conditions of \eqref{epsNMF} for $\bar{V}$ are
\begin{displaymath}
\bar{V}_{ik} \geq \epsilon, \quad  \alpha_{ik} \leq 1 \quad \textrm{ and } \quad (\bar{V}_{ik}-\epsilon)\,(\alpha_{ik} - 1)  = 0, \quad \forall i,k.
\end{displaymath}
Finally, by \eqref{limpoint}, we have either $\bar{V}_{ik} = \epsilon$ and $\alpha_{ik} \leq 1$, or $\bar{V}_{ik} > \epsilon$ and $\alpha_{ik} = 1$,$\forall i,k$. The same can be done for $\bar{W}$ by symmetry.
\end{proof} \vspace{0.2cm}

In order to implement the updates \eqref{NFmult}, one has to choose the matrices $P$ and $N$.
It is clear that $\forall P, N \geq \mathbf{0}$
such that $M = P-N$, there exists a matrix $C \geq \mathbf{0}$ such that the two components $P$ and $N$ can be written  $P = M_+ + C$ and $N = M_- + C$. When $C$ goes to infinity, the above updates do not change the matrices $V$ and $W$,
which seems to indicate that smaller values of $C$ are preferable. Indeed, in the case $r=1$, 
one can prove that $C=\mathbf{0}$ is an optimal choice: \vspace{0.1cm}

\begin{theorem} \label{r1opt} $\forall P,N \geq \mathbf{0}$ s.t. $M = P - N$, and $\forall v \in \mathbb{R}^{n}_+, w \in \mathbb{R}^{m}_+$:
\begin{equation} \label{r1opteq}
||M-v_1w||_F \leq ||M-v_2w||_F \leq ||M-vw||_F,
\end{equation}
for
\[ v_1  = v \circ \frac{[M_+ w^T]}{[v w w^T+M_- w^T]} \quad \textrm{ and } \quad v_2  =  v \circ \frac{[P w^T]}{[v w w^T+N w^T]}. \]
\end{theorem}
\begin{proof} The second inequality of \eqref{r1opteq} is a consequence of Theorem~\ref{NFmult}. For the first one, we treat the inequality separately for each entry of $v$ i.e.\@ we prove that
\[ ||M_{i:}-{v_1}_i w||_F \leq ||M_{i:}-{v_2}_iw||_F, \quad \forall i.\]
Let define $v_i^*$ as the optimal solution of the unconstrained problem i.e.
\[ v_i^* =  \textrm{argmin}_{v_i} ||M_{i:}-v_iw||_F = \frac{M_{i:} w^T}{ww^T},\]
and $a,b,d \geq 0$, $e > 0$, as
\[ a = (M_+)_{i:}w^T,\; b = (M_-)_{i:}w^T,\; d = (P - M_+)_{i:}w^T \; \textrm{ and } \; e = v_iww^T. \]
Noting that $P-M_+ = N-M_-$, we have
\[ v_i^* = v_i \, \Big( \frac{a-b}{ e } \Big),
\; {v_1}_i =  v_i \, \Big( \frac{a}{ e + b } \Big) \; \textrm{ and } \;
{v_2}_i =  v_i \, \Big( \frac{a + d }{ e+ b + d}\Big). \]
Suppose $v_i^* \geq v_i$. Therefore,
\[\frac{a-b}{ e } \geq 1 \; \Rightarrow \; a-b-e \geq 0 \; \Rightarrow \; \frac{a-b}{ e } \geq \frac{a}{ e + b }. \]
Moreover, $1 \leq \frac{a+d}{ e + b +d}$  since ${v_2}_i$ is a better solution than $v_i$ (Theorem~\ref{NFmult}). Finally,
\[1 \leq \frac{a+d}{ e + b + d} \leq \frac{a}{ e + b } \leq \frac{a-b}{ e } \quad \Rightarrow \quad v_i \leq {v_2}_i \leq {v_1}_i \leq v_i^*.\]
The case $v_i^* \leq v_i$ is similar.
\end{proof} \vspace{0.2cm}

Unfortunately, this result does not hold for $r>1$. This is even true for nonnegative matrices, i.e.\@ one can improve the effect of a standard
Lee and Seung multiplicative update by using a well-chosen matrix $C$.\\

\begin{example} With the following matrices
\begin{displaymath}
M = \left( \begin{array}{ccc}
     0  &   0 &    1 \\
     0    & 1  &   1\\
     1   &  1  &   0\\\end{array} \right), \;
V = \left( \begin{array}{cc}
     1  &   1\\
     1  &   0\\
     1  &   1\\\end{array} \right),\;
W = \left( \begin{array}{ccc}
     1    & 0  &   0 \\
     1   &  0 &    1\\\end{array} \right) \; \textrm{ and }  \;
C = \left( \begin{array}{ccc}
     0  &   0 &    1 \\
     0    & 0  &   0\\
     1   &  0  &   0\\\end{array} \right),
\end{displaymath}
we have $||M-V'W||_F < ||M-V''W||_F$ where $V'$ (resp. $V''$) is updated following \eqref{NFmult} using $P = M+C$ and $N=C$ (resp. $P = M$ and $N=\mathbf{0}$).
\end{example}
\noindent However, in practice, it seems that the choice of a proper matrix $C$ is nontrivial and cannot accelerate significantly the speed of convergence.

\section{How good are the Multiplicative Updates (MU) of Lee and Seung?} \label{interMU}

In this section, we use Theorem~\ref{NFmultTh} to interpret the multiplicative rules for \eqref{NMF} and show why the HALS algorithm performs much better in practice.

\subsection{An Improved Version of the MU} The aim of the MU is to improve a current solution $(V,W) \geq \mathbf{0}$ by optimizing alternatively $V$ ($W$ fixed), and vice-versa. In order to prove the monotonicity of the MU, $||M-VW||_F$ was shown to be nonincreasing under an update of a single row of $V$ (resp. column of $W$) since the objective function can be split into $m$ (resp. $n$) independent quadratic terms, each depending on the entries of a row of $V$ (resp. column of $W$); cf. proof of Theorem~\ref{NFmultTh}.

However, there is no guarantee, a priori, that the algorithm is also nonincreasing with respect to an individual update of a column of $V$ (resp. row of $W$). In fact, each entry of a column of $V$ (resp. row of $W$) depends on the other entries of the same row (resp. column) in the cost function. The next theorem states that this property actually holds.
\begin{corollary}  \label{corLS} For $V, W \geq \mathbf{0}$, $||M-VW||_F$ is nonincreasing under
\begin{equation} \label{corLSupdt}
V_{:k} \leftarrow V_{:k} \circ \frac{[M W_{k:}^T]}{[VWW_{k:}^T]},
\qquad  W_{k:} \leftarrow W_{k:} \circ \frac{[V_{:k}^T M]}{[V_{:k}^T VW]}, \quad \forall k,
\end{equation}
i.e. under the update of any column of $V$ or any row of $W$ using the MU \eqref{LSupdate}.
\end{corollary}
\begin{proof} This is a consequence of Theorem~\ref{NFmultTh} using $P = M$ and $N = \sum_{i\neq k} V_{:i} W_{i:}$. \\ In fact, $||M-VW||_F = ||(M - \sum_{i\neq k} V_{:i} W_{i:}) - V_{:k} W_{k:}||_F.$
\end{proof} \vspace{0.2cm}

Corollary \ref{corLS} sheds light on a very interesting fact: the multiplicative updates are also trying to optimize alternatively the columns of $V$ and the rows of $W$ using a specific cyclic order: first, the columns of $V$ and then the rows of $W$.
We can now point out two ways of improving the MU:
\begin{enumerate}
\item When updating a column of $V$ (resp. a row of $W$), the columns (resp. rows) already updated are not taken into account: the algorithm uses their old values;
\item The multiplicative updates are not optimal: $P \neq M_+$ and $N \neq M_-$ (cf. Theorem~\ref{r1opt}). Moreover, there is actually a closed-form solution for these subproblems (cf. HALS algorithm, Section~\ref{secHALS}).
\end{enumerate}
Therefore, using Theorem~\ref{r1opt}, we have the following new improved updates
\begin{corollary} \label{NFr1} For $V, W \geq \mathbf{0}$, $||M-VW||_F$ is nonincreasing under
\begin{equation} \label{corNF}
V_{:k} \leftarrow V_{:k} \circ \frac{[(R_{k})_+ W_{k:}^T]}{[V_{:k} W_{k:} W_{k:}^T+(R_{k})_- W_{k:}^T]}, \quad
W_{k:} \leftarrow W_{k:} \circ \frac{[V_{:k}^T (R_{k})_+]}{[V_{:k}^T V_{:k} W_{k:}+V_{:k}^T (R_{k})_-]}, \; \forall k,
\end{equation}
with $R_k = M - \sum_{i\neq k} V_{:i} W_{i:}$. Moreover, the updates \eqref{corNF} perform better than the updates \eqref{corLSupdt}, but worse than the updates (\ref{gilVcol1}-\ref{gilVcol}) which are optimal.
\end{corollary}
\begin{proof} This is a consequence of Theorem~\ref{NFmultTh} and \ref{r1opt} using $P = (R_k)_+$ and $N = (R_k)_-$.\\ In fact, $||M-VW||_F = ||R_k - V_{:k} W_{k:}||_F$.
\end{proof}

\subsection{An example}

Figure \ref{cbcl} shows an example of the behavior of the different algorithms: the original MU (Section~\ref{LSalgo}), the improved version (Corollary \ref{NFr1}) and the \emph{optimal} HALS method (Section~\ref{secHALS}). The test was carried out on a commonly used data set for NMF: the cbcl face database\footnote{CBCL Face Database $\#1$, MIT Center For Biological and Computation Learning. \\ Available at $http$:$//cbcl.mit.edu/cbcl/software$-$datasets/FaceData2.html$.}; 2429 faces (columns) consisting each of $19 \times 19$ pixels (rows) for which we set $r=40$ and we used the same \emph{scaled} (see Remark~\ref{scaled} below) random initialization and the same cyclic order (same as the MU i.e.\@ first the columns of $V$ then the rows of $W$) for the three algorithms.
\begin{figure}[ht]
\begin{center}
\includegraphics[width=7cm]{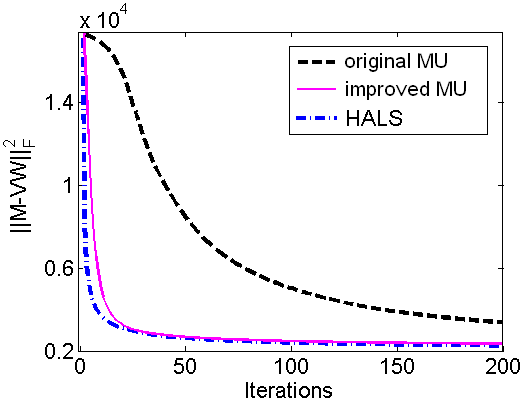}
\caption{Comparison of the MU of Lee and Seung \eqref{LSupdate}, the optimal rank-one NF multiplicative updates \eqref{corNF} and the HALS (\ref{gilVcol1}-\ref{gilVcol}) applied to the cbcl face database.}
\label{cbcl}
\end{center}
\end{figure}
We observe that the MU converges significantly less rapidly than the two other algorithms. There do not seem to be good reasons to use either the MU or the method of Corollary \ref{NFr1} since there is a closed-form solution (\ref{gilVcol1}-\ref{gilVcol}) for the corresponding subproblems.

Finally, the HALS algorithm has the same computational complexity \cite{diep} and performs provably much better than the popular multiplicative updates of Lee and Seung.
Of course, because of the NP-hardness of \eqref{NMF} and the existence of numerous locally optimal solutions, it is not possible to give a theoretical guarantee that HALS will converge to a better solution than the MU: although its iterations are \emph{locally} more efficient, they could still end up at a worse local optimum.
\vspace{0.2cm}

\begin{remark} \label{power}
For $r=1$, one can check that the three algorithms above are equivalent for \eqref{NMF}. Moreover, they correspond to the power method \cite{Gil} which converges to the optimal rank-one solution, given that it is initialized with a vector which is not perpendicular to the singular vector corresponding to the maximum singular value.
\end{remark} \vspace{0.1cm}

\begin{remark} \label{scaled}
We say that $(V,W)$ is scaled if the optimal solution to the problem
\begin{equation} \label{sca}
\min_{\alpha \in \mathbb{R}}||M-\alpha VW||_F
\end{equation}
is equal to 1. Obviously, any stationary point is scaled~; the next Theorem is an extension of a result of Ho et al. \cite{ho}.
\end{remark}
\newpage
\begin{theorem} \label{scth}
The following statements are equivalent
\begin{itemize}
\item[(1)] $(V,W)$ is scaled;
\item[(2)] $VW$ is on the boundary of $\mathcal{B}\Big(\frac{M}{2},\frac{1}{2}||M||_F\Big)$, the ball centered at $\frac{M}{2}$ of radius $\frac{1}{2}||M||_F$; \vspace{0.05cm}
\item[(3)] $||M - VW||_F^2 = ||M||_F^2 - ||VW||_F^2$ (and then $||M||_F^2 \geq ||VW||_F^2$).
\end{itemize}
\end{theorem}
\begin{proof} The solution of \eqref{sca} can be written in the following closed form
\begin{equation}
\alpha = \frac{\left\langle M,VW\right\rangle}{\left\langle VW,VW\right\rangle},
\end{equation}
where $\left\langle A,B\right\rangle = \sum_{ij} A_{ij}B_{ij} = trace(AB^T)$ is the scalar product associated with the Frobenius norm. Since $\alpha = 1$,
\begin{eqnarray*}
\left\langle VW-M,VW\right\rangle & = & 0 \\
\left\langle VW-M,VW\right\rangle + \left\langle \frac{M}{2},\frac{M}{2}\right\rangle
& =& \left\langle \frac{M}{2},\frac{M}{2}\right\rangle \\
\left\langle\frac{M}{2}-VW,\frac{M}{2}-VW\right\rangle & =& \left\langle  \frac{M}{2},\frac{M}{2}\right\rangle,
\end{eqnarray*}
so that (1) and (2) are equivalent. For the equivalence of (1) and (3), we have
\begin{eqnarray*}
\left\langle M-VW,M-VW\right\rangle & =& ||M||_F^2 - 2 \left\langle M,VW \right\rangle + ||VW||_F^2 \\
& =& ||M||_F^2 - ||VW||_F^2 - 2\Big( \left\langle M,VW \right\rangle -  \left\langle VW,VW \right\rangle \Big).\\
& =& ||M||_F^2 - ||VW||_F^2
\end{eqnarray*}
if and only if $\left\langle M,VW \right\rangle = \left\langle VW,VW\right\rangle$.
\end{proof} \vspace{0.2cm}

\noindent Theorem~\ref{scth} can be used as follows: when you compute the error of the current solution, you can scale it without further computational cost. In fact,
\begin{eqnarray}
||M-VW||_F^2 &= & \left\langle M-VW,M-VW\right\rangle \nonumber \\
& =& ||M||_F^2 - 2 \left\langle M,VW \right\rangle + ||VW||_F^2. \label{error}
\end{eqnarray}
Note that the third term of \eqref{error} can be computed in $O(\max(m,n)r^2)$ operations since
\begin{eqnarray*}
||VW||_F^2 & = & \sum_{ij}  \Big( \sum_k  V_{ik}^2 W_{kj}^2 \Big) + 2 \sum_{ij}  \Big( \sum_{k \neq l} V_{ik} W_{kj} V_{il} W_{lj} \Big) \\
		       & = & \sum_k \Big(\sum_{i} V_{ik}^2\Big) \Big(\sum_j W_{kj}^2\Big) + 2 \sum_{k \neq l} \Big(\sum_{i} V_{ik} V_{il}\Big)   \Big(\sum_j  W_{kj} W_{lj}  \Big) \\
		       & = & ||(V^TV) \circ (WW^T)||_1
\end{eqnarray*}
where $||A||_1 = \sum_{ij} |A_{ij}|$.
This is especially interesting for sparse matrices since only a small number of the entries of $VW$ (which could be dense) need to be computed to evaluate the second term of \eqref{error}.

\section{Biclique Finding Algorithm} \label{mbfa}

In this section, an algorithm for the maximum edge biclique problem whose main iteration requires $O(|E|)$ operations is presented. It is based on the multiplicative updates for Nonnegative Factorization and the strong relation between these two problems (Theorems \ref{thp}, \ref{th3v} and \ref{Sdinf}). We compare the results with other algorithms with iterates requiring $O(|E|)$ operations using the DIMACS database and random graphs.

\subsection{Description}

For $d$ sufficiently large, stationary points of \eqref{NF1} are close to bicliques of \eqref{MBP} (Theorem \ref{Sdinf}). Moreover, the two problems have the same cost function. One could then think of applying an algorithm that finds stationary points of \eqref{NF1} in order to localize a large biclique of the graph generated by $M_b$.
This is the idea of Algorithm \ref{MBFA} using the multiplicative updates \eqref{NFmult}
with
\[ P = (M_d)_+ = M_b \quad \textrm{ and } \quad N = (M_d)_- = d(\mathbf{1}_{m \times n} - M_b). \]
A priori, it is not clear what value $d$ should take. Following the spirit of homotopy methods, we chose to start the algorithm with a small value of $d$ and then to increase it until the algorithm converges to a biclique of \eqref{MBP}.
\algsetup{indent=2em}
\begin{algorithm}[h!]
\caption{Biclique Finding Algorithm in $O(|E|)$ operations}\label{MBFA}
\begin{algorithmic}[1]
\REQUIRE $M_b \in \{0,1\}^{m \times n}$, $v \in \mathbb{R}^{m}_{++}$, $w \in \mathbb{R}^{n}_{++}$, $d = d_0 > 0$, $\alpha > 1$.
\medskip
\FOR {$k = 1, 2, \dots$}
\STATE
\begin{eqnarray}
 v  & \leftarrow &  v^{} \, \circ \frac{[M_b{w^{}}^T]}{[v^{}||w^{}||_2^2+d (\mathbf{1}_{m} ||w^{}||_1 - M_b{w^{}}^T)]}   \label{a} \\
 w & \leftarrow & w^{} \circ \frac{[{v}^TM_b]}{[||v^{}||_2^2 w^{}+d (\mathbf{1}_{n} ||v^{}||_1 - {v^{}}^TM_b)]}  \label{b} \\
 d \; & = & \, \alpha d \nonumber
 \end{eqnarray}
\ENDFOR
\end{algorithmic}
\end{algorithm}

We observed that initial value of $d$ should not be chosen too large: otherwise, the algorithm often converges to the trivial solution: the empty biclique. In fact, in that case, the denominators in \eqref{a} and \eqref{b} will be large, even during the initial steps of the algorithm, and the solution is then forced to converge to zero. Moreover, since the denominators in \eqref{a} and \eqref{b} depend on the graph density, the denser the graph is, the greater $d_0$ can be chosen and vice versa. On the other hand, since our algorithm is equivalent to the power method for $d=0$ (cf. Remark \ref{power}), if $d_0$ is chosen too small, it will converge to the same solution: the one initialized with the best rank-one approximation of $M_b$.

For the stopping criterion, one could, for example, wait until the rounding of $vw$ coincides with a feasible solution of \eqref{MBP}.\\

\subsection{Other Algorithms in $O(|E|)$ operations}

We briefly present here two other algorithms to find maximal bicliques using $O(|E|)$ operations per iteration.

\subsubsection{Motzkin-Strauss Formalism}

In \cite{Ding}, the generalized Motzkin-Strauss formalism for cliques is extended to bicliques by defining the optimization problem
\begin{displaymath}
\max_{\mathbf{x} \in F_x^{\alpha}, \mathbf{y} \in F_y^{\beta}}  \mathbf{x}^T M_b \, \mathbf{y}
\end{displaymath}
where
$F_x^{\alpha} = \{ x \in \mathbb{R}^n_+ | \sum_{i=1}^n x_i^{\alpha} = 1\}$,
$F_y^{\beta} = \{ y \in \mathbb{R}^n_+ | \sum_{i=1}^n y_i^{\beta} = 1\}$ and $\alpha, \beta \ll 1$.\\
Nonincreasing multiplicative updates for this problem are then provided:
\begin{displaymath}
\mathbf{x}  \leftarrow  \Big( \mathbf{x} \circ \frac{M_b \, \mathbf{y}}{\mathbf{x}^T M_b \, \mathbf{y}} \Big)^{\frac{1}{\alpha}}, \quad
\mathbf{y}  \leftarrow  \Big( \mathbf{y} \circ \frac{M_b^T \mathbf{x}}{\mathbf{x}^T M_b \, \mathbf{y}} \Big)^{\frac{1}{\beta}}
\end{displaymath}
This algorithm does not necessarily converge to a biclique: if $\alpha$ and $\beta$ are not sufficiently small, it may converge to a dense bipartite subgraph (a bicluster). In fact, for $\alpha=\beta=2$, it converges to an optimal rank-one solution of the unconstrained problem as our algorithm does for $d=0$. In \cite{Ding}, it is suggested to use $\alpha$ and $\beta$ around 1.05. Finally, $\alpha \neq \beta$ will favor one side of the biclique. We will use $\alpha = \beta$.

\subsubsection{Greedy Heuristic}

The simplest heuristic one can imagine is to add, at each step, a vertex which is connected to most vertices in the other side of the bipartite graph. Each time a vertex is selected, the next choices are restricted in order to get a biclique eventually: the vertices which are not connected to the one you have just chosen are deleted. The procedure is repeated on the remaining graph until you get a biclique. One can check that this produces a maximal biclique.

\subsection{Results} \label{results}

We first present some results for graphs from the DIMACS graph dataset\footnote{$ftp$:$//dimacs.rutgers.edu/pub/challenge/graph/benchmarks/clique$.}.
We extracted bicliques in those (not bipartite) graphs using the preceding algorithms.
We performed 100 runs, 200 iterations each, for the two algorithms with the same initializations.
We tried to choose appropriate parameters\footnote{For the M.-S. algorithm: the alternative choices of parameters $\alpha= 1.005$ and $\alpha= 1.05$ for the DIMACS graphs and $\alpha= 1.005$ and $\alpha= 1.1$ for the random graphs were tested and all gave worse results. Small changes to the parameters of the Mult.\@ algorithm led to similar results, so that it seems less sensitive to the choice of its parameters than the M.-S. algorithm.} for both algorithms. Table \ref{tableDIMACS} displays the cardinality of the biclique extracted by the different algorithms.

\noindent Table \ref{tableRand} shows the results for random graphs: we have generated randomly 100 graphs with 100 vertices for different densities (the probability of an edge to belong to the graph is equal to the density). The average numbers of edges in the solutions for the different algorithms are displayed for each density. We kept the same configuration as for the DIMACS graphs (same initializations, 100 runs for each graph, 200 iterations).
\normalsize It seems that the multiplicative updates generates, in general, better solutions, especially when dealing with dense graphs. The algorithm based on the Motzkin-Strauss formalism seems less efficient and is more sensitive to the choice of its parameters.

\scriptsize
\begin{center}
\begin{table}[h!]
\begin{center}
\begin{tabular}{|c|c|c|c|c|}
\hline
& size  & Greedy      &  M.-S.   & Mult.   \\
& (m) & & ($\alpha = 1.01$) & ($d=1$, $\alpha=1.1$) \\
\hline
\begin{tabular}{c} \\ham62\\ ham64\\ ham82\\ ham84\\ john824\\ john844\\ john1624\\ john3224\\ MANN a9\\ MANN a27
\end{tabular}
&
\begin{tabular}{c} \\64\\ 64\\ 256\\ 256\\ 28\\ 70\\ 120\\ 496\\ 45\\378
\end{tabular}
&
\begin{tabular}{c} \\304\\ 42\\ 4672\\ 440\\ 36\\ 182\\ 784\\ 14400\\ 289\\28728
\end{tabular}
&
\begin{tabular}{c|c} mean & best \\ \hline 157 & 225\\ 21 & 36 \\ 2839 & 3920\\ 226 & 506\\ 26 & 36\\ 132 & 225\\ 457 & 700\\ 6294 & 9129 \\ 272 & 342\\15946 & 28875
\end{tabular}
&
\begin{tabular}{c|c} mean & best \\\hline 269 & 320\\ 37 & 42 \\ 4569  &  4770\\ 830  & 1015\\ 28 & 36\\ 220 & 225\\
514 & 675\\ 8722 & 9108 \\ 342 & 342\\30800 & 30800 \\
\end{tabular}\\
\hline
\end{tabular}
\end{center}
\caption{Solutions for DIMACS data: number of edges in the bicliques.}
\label{tableDIMACS}
\end{table}
\end{center}

\scriptsize
\begin{center}
\begin{table}[h!]
\begin{center}
\begin{tabular}{|c|c|c|c|c|}
\hline
density	& Greedy      &  M.-S.  &  M.-S.  & Mult.    \\
 & & ($\alpha = 1.01$) & ($\alpha = 1.05$) & ($d=1$, $\alpha=1.1$)\\
\hline
\begin{tabular}{c} \\0.1\\ 0.2\\ 0.3\\ 0.4\\ 0.5\\ 0.6\\ 0.7\\ 0.8\\ 0.9
\end{tabular}
&
\begin{tabular}{c} \\8.9\\ 15.8\\ 24.2\\ 37.7\\ 58.8\\ 92.5\\ 158.4\\ 311.0\\ 806.8
\end{tabular}
&
\begin{tabular}{c|c} mean & best \\ \hline 10.4 & 18.9\\ 14.5 & 29.8 \\ 18.9 & 38.7\\ 24.0 & 51.2\\ 31.7 & 69.0\\ 43.4 & 93.7\\ 63.6 & 133.4\\ 106.0 & 207.1 \\ 241.9 & 431.3
\end{tabular}
&
\begin{tabular}{c|c} mean & best \\ \hline 14.4 & 19.2\\ 23.7 & 31.5 \\ 33.8 & 43.4\\ 46.9 & 61.3\\ 65.8 & 86.7\\ 90.5 & 127.9\\ 117.9 & 190.0\\ 154.4 & 261.5 \\ 88.1 & 235.4
\end{tabular}
&
\begin{tabular}{c|c} mean & best \\ \hline 14.4 & 19.2\\ 23.9 & 31.5 \\ 34.1 & 43.3\\ 47.0 & 61.0\\ 67.6 & 87.0\\ 101.7 & 127.8\\ 172.2 & 202.4\\ 328.0 & 342.3 \\ 828.1 & 828.1 \\
\end{tabular}  \\ \hline
\end{tabular}
\end{center}
\caption{Solutions for random graphs: average number of edges in the bicliques.
}
\label{tableRand}
\end{table}
\end{center}
\normalsize

\begin{remark} Algorithm \ref{MBFA} enjoys some flexibility:
\begin{itemize}
\item It is applicable to non-binary matrices i.e.\@ weighted graphs. 
\item It is possible to favor one side of the biclique. In fact, the multiplicative updates for NF can be adapted using the same developments as in Section \ref{MUNF} to cost functions with regularization terms, e.g.
\begin{displaymath}
\min_{v,w \geq 0} \quad ||M-vw||_F^2 + \alpha ||v||_2^2 + \beta ||w||_2^2.
\end{displaymath}
\item If $d$ is kept sufficiently small, for example replacing $d=\alpha d$ by $d = \min(\alpha d, d_m)$ for some $d_m > 0$, there is no guarantee that the algorithm will converge to a biclique.
However, the negative entries in $M_d$ will enforce the corresponding entries of the solutions of \eqref{NF1} to be small (recall that Theorem~\ref{Sdinf} states that, for $d$ sufficiently large, they will be equal to zero). Therefore, by rounding these solutions, instead of a biclique, one gets a dense submatrix of $M_b$ \@ i.e. a \emph{bicluster}. Algorithm~\ref{MBFA} can then be used as a {biclustering algorithm}. The density of the corresponding submatrix will depend on the choice of $d_m$. Table \ref{tableClass} gives an example of such behavior.
\begin{center}
\begin{table}[h!]
\begin{center}
\begin{tabular}{|c|c|c|c|c|c|c|}
\hline
$d_m$     & 0.01 & 0.05 & 0.1 & 0.5 & 1 & 1.5 \\ \hline
size    & 5412 & 4428 & 2952 & 1073 & 595 & 539\\ \hline
density & $29\%$ & $31\%$ & $35\%$ & $42\%$ & $51\%$ & $52\%$\\
\hline
\end{tabular}
\end{center}
\caption{Biclusters for the 'classic' text mining dataset (7094 texts and 41681 words with more than $99.9\%$ of entries equal to zero) with parameters $d_0 = 10^{-5}, \alpha = 1.025, \textrm{maxiter} = 500$.}
\label{tableClass}
\end{table}
\end{center}
\end{itemize}
\end{remark}

\section{Conclusion}

We have introduced Nonnegative Factorization (NF), a new variant of Nonnegative Matrix Factorization (NMF), and proved its NP-hardness for any fixed rank by reduction to the maximum edge biclique problem. The multiplicative updates for NMF can be generalized to NF and provide a new interpretation of the algorithm of Lee and Seung, which explains why it does not perform well in practice. We also developed an heuristic algorithm for the biclique problem whose iterations require $O(|E|)$ operations, based on theoretical results about stationary points of a specific rank-one nonnegative factorization problem \eqref{NF1} and the use of multiplicative updates.

To conclude, we point out that none of the algorithms presented in this paper is guaranteed to converge to a globally optimal solution (and, to the best of our knowledge, such an algorithm has not been proposed yet) ; this is in all likelihood due to the NP-hardness of the NMF and NF problems. Indeed, only convergence to a stationary point has been proved for the algorithms of Sections~\ref{nmf} and~\ref{SecNF}, a property which, while desirable, provides no guarantee about the quality of the solution obtained (for example, nothing prevents these methods from converging to a stationary but rank-deficient solution, which in most cases could be further improved). Finally, no convergence proof for the biclique finding algorithm introduced in Section~\ref{mbfa} is provided (convergence results from the preceding sections no longer hold because of the dynamic updates of parameter $d$) ; however, this heuristic seems to give very satisfactory results in practice.\\

\noindent \textbf{Acknowledgment.} \;
We thank Pr. Paul Van Dooren and Pr. Laurence Wolsey for helpful discussions and advice.

\end{document}